\documentclass[11pt,a4paper,reqno]{amsart}

\usepackage[T1]{fontenc}
\usepackage[english]{babel}
\usepackage{mathtools}
\usepackage{amssymb}
\usepackage{libertinus}
\usepackage[cal=boondoxo,bb=ams]{mathalfa}
\usepackage{microtype}
\usepackage{enumitem}
\usepackage{pgfplots}
\pgfplotsset{compat=1.18}
\usepgfplotslibrary{fillbetween}
\usepackage{etoolbox}
\usepackage{needspace}
\usepackage{float}
\usepackage{xcolor}

\allowdisplaybreaks[2]
\usepackage[
  a4paper,
  left=30mm,
  right=30mm,
  top=27mm,
  bottom=30mm,
  headsep=8mm,
  footskip=13mm,
  heightrounded
]{geometry}

\setlist[itemize]{leftmargin=2.1em,itemsep=0.25em,topsep=0.45em,parsep=0pt}
\setlist[enumerate]{leftmargin=2.1em,itemsep=0.25em,topsep=0.45em,parsep=0pt}

\numberwithin{equation}{section}

\makeatletter
\def\@settitle{%
  \begin{center}
    \vspace*{-0.4em}
    {\normalfont\bfseries\fontsize{17}{21}\selectfont\@title\par}
    \vspace{0.7em}
  \end{center}}
\renewcommand\section{\@startsection{section}{1}{\z@}%
  {1.5\baselineskip plus 0.2\baselineskip minus 0.1\baselineskip}%
  {0.65\baselineskip}{\normalfont\Large\bfseries}}
\renewcommand\subsection{\@startsection{subsection}{2}{\z@}%
  {1.15\baselineskip plus 0.2\baselineskip minus 0.1\baselineskip}%
  {0.45\baselineskip}{\normalfont\large\bfseries}}
\renewcommand\subsubsection{\@startsection{subsubsection}{3}{\z@}%
  {0.9\baselineskip plus 0.15\baselineskip minus 0.1\baselineskip}%
  {0.35\baselineskip}{\normalfont\normalsize\bfseries}}
\makeatother

\makeatletter
\patchcmd{\@setauthors}{\centering\footnotesize}{\centering\normalsize}{}{}
\patchcmd{\@setauthors}{\MakeUppercase{\authors}}{\authors}{}{}
\patchcmd{\maketitle}{\uppercasenonmath\shorttitle}{}{}{}
\patchcmd{\maketitle}{\@nx\MakeUppercase{\the\toks@}}{\the\toks@}{}{}
\makeatother

\usepackage{hyperref}
\hypersetup{
  hidelinks,pdfencoding=auto,
  pdftitle={Sharp lifespan estimates for Nakao's problem in one space dimension},
  pdfauthor={Wenhui Chen},
  pdfkeywords={Nakao's problem, damped wave equation, blow-up, sharp lifespan}
}
\usepackage[nameinlink,capitalise,noabbrev]{cleveref}

\theoremstyle{plain}
\newtheorem{theorem}{Theorem}[section]
\newtheorem{lemma}[theorem]{Lemma}
\newtheorem{prop}[theorem]{Proposition}

\theoremstyle{definition}
\newtheorem{defn}{Definition}[section]

\theoremstyle{remark}
\newtheorem{remark}{Remark}[section]

\newcommand{\ml}{\mathcal}
\newcommand{\mb}{\mathbb}
\newcommand{\dd}{\,\mathrm{d}}
\DeclareMathOperator{\supp}{supp}

\DeclareMathOperator{\lin}{lin}
\DeclareMathOperator{\nlin}{nl}

\title[Sharp lifespan estimates for Nakao's problem in 1D]
  {Sharp lifespan estimates for Nakao's problem\\ in one space dimension}
\author[W. Chen]{Wenhui Chen}
\address{School of Mathematics and Information Science, Guangzhou University,
  Guangzhou 510006, P. R. China}
\email{wenhui.chen.math@gmail.com}
\date{}

\begin{document}

\raggedbottom
\begin{abstract}
  In this paper, we consider Nakao's problem in one space dimension, which couples a semilinear damped wave equation with a semilinear wave equation.
  For compactly supported energy data satisfying suitable integral sign conditions, we establish sharp lifespan estimates for all powers $p,q>1$.
  The estimates are classified according to whether the total integral of the initial speed in the wave equation is positive or zero.
\end{abstract}
\maketitle

\section{Introduction}

In this paper, we consider Nakao's problem in one space dimension, namely, the weakly coupled Cauchy problem
\begin{align}
  \begin{cases}
    u_{tt}-u_{xx}+u_t=|v|^p,&(t,x)\in(0,T_{\varepsilon})\times\mb R,\\
    v_{tt}-v_{xx}=|u|^q,&(t,x)\in(0,T_{\varepsilon})\times\mb R,\\
    (u,u_t)(0,x)=\varepsilon(u_0,u_1)(x),&x\in\mb R,\\
    (v,v_t)(0,x)=\varepsilon(v_0,v_1)(x),&x\in\mb R,
  \end{cases}
  \label{Eq-Main-System}
\end{align}
where $p,q>1$, $\varepsilon>0$ describes the size of the initial data, and $T_{\varepsilon}$ denotes the lifespan of the local in-time solution.
Under suitable sign and nontriviality assumptions on the initial data, finite-time blow-up occurs for every $p,q>1$ in one space dimension \cite{Wakasugi=2017,Chen-Reissig=2021,Kita-Kusaba=2022}.
Our purpose is to establish sharp estimates for $T_{\varepsilon}$ for sufficiently small $\varepsilon>0$ and to determine how the total integral of the initial speed $v_1$ affects these estimates.

Let us first recall the sharp lifespan estimates for the corresponding single semilinear equations in one space dimension.
\begin{itemize}
  \item For the semilinear wave equation
    \begin{align}
      \begin{cases}
        w_{tt}-w_{xx}=|w|^r,&(t,x)\in(0,T^{\mathrm W}_{\varepsilon})\times\mb R,\\
        (w,w_t)(0,x)=\varepsilon(f,g)(x),&x\in\mb R,
      \end{cases}
      \label{Eq-Scalar-Wave}
    \end{align}
    with $r>1$ and fixed, smooth, compactly supported data $(f,g)\not\equiv(0,0)$, \cite[Theorems~1.1 and~1.2]{Zhou=1992} gives
    \begin{align}
      T^{\mathrm W}_{\varepsilon}\approx
      \begin{cases}
        \varepsilon^{-\frac{r-1}{2}}&\text{if}\ \ \displaystyle\int_{\mb R}g(x)\dd x\neq0,\\[10pt]
        \varepsilon^{-\frac{r(r-1)}{r+1}}&\text{if}\ \ \displaystyle\int_{\mb R}g(x)\dd x=0,
      \end{cases}
      \label{Eq-Scalar-Wave-Lifespan}
    \end{align}
    for sufficiently small $\varepsilon>0$, where the implicit constants are independent of $\varepsilon$.
    These estimates use the d'Alembert formula for the lower bounds and nonlinear integral inequalities for the upper bounds.
    If the total integral of the initial speed is zero, the free wave satisfies Huygens' principle, which leads to the longer lifespan in the second case.
  \item For the semilinear damped wave equation
    \begin{align}
      \begin{cases}
        z_{tt}-z_{xx}+z_t=|z|^r, &(t,x)\in(0,T^{\mathrm D}_{\varepsilon})\times\mb R,\\
        (z,z_t)(0,x)=\varepsilon(f,g)(x),&x\in\mb R,
      \end{cases}
      \label{Eq-Scalar-Damped}
    \end{align}
    with $r>1$ and fixed, nonnegative, smooth, compactly supported data $(f,g)\not\equiv(0,0)$, the corresponding results are
    \begin{align*}
      T^{\mathrm D}_{\varepsilon}
      \begin{cases}
        \approx \varepsilon^{-\frac{2(r-1)}{3-r}}&\text{if}\ \ 1<r<3,\\
        \approx\exp\left(c\varepsilon^{-2}\right)&\text{if}\ \ r=3,\\
        =\infty&\text{if}\ \ r>3,
      \end{cases}
    \end{align*}
    for sufficiently small $\varepsilon>0$, see \cite{Li-Zhou=1995,Ikeda-Sobajima-Wakasugi=2019}.
    Here, the decay estimates for the damped wave equation \cite{Marcati-Nishihara=2003} enter the existence argument, while integral inequalities give the blow-up estimates.
    The critical exponent $r=3=1+\frac{2}{n}\big|_{n=1}$ agrees with the well-known Fujita exponent for the semilinear heat equation in one space dimension.
\end{itemize}
The lifespan estimates for the problem \eqref{Eq-Main-System} do not follow directly from those for these single semilinear equations.
Only the $u$-equation contains a frictional damping term $u_t$, whose influence on $v$ is transmitted through the weakly coupled nonlinear terms.
Thus, the diffusion behavior of $u$ and the propagation of $v$ have to be estimated together, although the scalar lifespan arguments use different properties of their linear solutions.

The different critical curves for weakly coupled systems of semilinear wave equations \cite{Del-Santo=1997,Del-Santo-Georgiev-Mitidieri=1997, Del-Santo-Mitidieri=1998,Agemi-Kurokawa-Takamura=2000, Kurokawa-Takamura=2003,Georgiev-Takamura-Zhou=2006, Kurokawa-Takamura-Wakasa=2012,Ikeda-Sobajima-Wakasa=2019} and for weakly coupled systems of semilinear damped wave equations \cite{Sun-Wang=2007,Narazaki=2009,Nishihara-Wakasugi=2014,Chen-Dao=2023} motivated the question of the critical relation between $p$ and $q$ when damping acts in only one equation (namely, a partial damping).
This question is known as Nakao's problem, named after Mitsuhiro Nakao (Kyushu University), who studied related initial-boundary value problems in \cite{Nakao=2016,Nakao=2018}.

For the Cauchy problem associated with Nakao's problem, a blow-up result with a parabolic-like condition on the powers was obtained in \cite{Wakasugi=2017} by a test function argument using space-time cutoffs.
By using an iteration method with a slicing procedure for unbounded exponential multipliers, the work \cite{Chen-Reissig=2021} derived a wave-like blow-up condition and upper lifespan estimates.
The test function argument in \cite{Kita-Kusaba=2022}, based on special solutions to the linear wave equation, covers the union of the corresponding subcritical blow-up ranges and yields upper lifespan estimates under integral sign assumptions.
The first small data global in-time existence result in three space dimensions was established in \cite{Georgiev-Kita=2026} by weighted pointwise estimates.
Subsequently, global in-time existence in a Sobolev framework in two and three space dimensions was proved in \cite{Chen-Global=2026}, together with scattering of the undamped component to a free wave.
Related Nakao-type problems have been studied for derivative-type nonlinearities \cite{Chen=2022,Palmieri-Takamura=2023,Liu=2026}, additional time-dependent damping terms \cite{Li-Palmieri=2025,Li-Palmieri-Derivative=2025}, and compact Lie groups \cite{Chen-Palmieri=2026}.

For compactly supported initial data, let us introduce
\begin{align}
  M_u:=\int_{\mb R}\bigl(u_0(x)+u_1(x)\bigr)\dd x \ \ \text{and}\ \ M_v:=\int_{\mb R}v_1(x)\dd x,
  \label{Eq-Initial-Integrals}
\end{align}
and define
\begin{align}
  \gamma_{\neq0}(p,q)&:= \max\left\{\frac{q+2}{pq-1},\, \frac{3-q}{2q}+\frac{q+2}{q(pq-1)}\right\},
  \label{Eq-Gamma}\\
  \gamma_{=0}(p,q)&:= \max\left\{\frac{2p+1}{p(pq-1)},\, \frac{3-q}{2q}+\frac{q+2}{q(pq-1)}\right\}.
  \label{Eq-Gamma-Zero}
\end{align}
Under the assumptions of Theorems~\ref{Theorem-Main} and \ref{Theorem-Zero}, respectively, we prove that
\begin{align}
  T_{\varepsilon}\approx
  \begin{cases}
    \varepsilon^{-\frac{1}{\gamma_{\neq0}(p,q)}} &\text{if}\ \ M_v>0,\\[3pt]
    \varepsilon^{-\frac{1}{\gamma_{=0}(p,q)}} &\text{if}\ \ M_v=0.
  \end{cases}
  \label{Eq-Intro-Sharp}
\end{align}
Under the common initial data assumptions in Remark~\ref{Remark-Lifespan-Comparison}, partial damping yields a lifespan strictly between those of the corresponding undamped and fully damped systems for sufficiently small $\varepsilon>0$.
For the upper bounds in \eqref{Eq-Intro-Sharp}, we use spatial integral inequalities and a comparison with an ordinary differential system.
For the lower bounds, we combine the two Duhamel estimates and construct time weights reflecting the different contributions of the wave and damped components.
When $M_v>0$, we estimate the full wave component, retaining both the contribution of the free wave and that generated by the linear damped component.
When $M_v=0$, we estimate the nonlinear part separately and exploit the representation of the free wave as two travelling waves to sharpen the Duhamel estimate for the damped component.
Estimating $Q(T)$ defined in \eqref{Eq-Lower-Q} with these weights gives two restrictions on $T$ ensuring that $Q(T)$ is sufficiently small, and hence the sharp lower lifespan bounds.

\medskip\paragraph{Notation.}
We write $L^r:=L^r(\mb R)$ and $H^1:=H^1(\mb R)$, and set $\langle t\rangle:=1+t$ for $t\geqslant0$.
The constants $c,C>0$ may change from line to line and depend on $p,q$, the fixed initial data and their support radius, but are independent of $t$ and $\varepsilon$ unless otherwise indicated.
We write $f\lesssim g$ if $f\leqslant Cg$ for such a constant, and $f\approx g$ if $f\lesssim g$ as well as $g\lesssim f$.
For a finite family of functions, $\supp(f_1,\ldots,f_k)$ denotes the union of their supports.

\section{Main results}\label{Section-Main-Result}

For some $R>0$, let the real-valued initial data, independent of $\varepsilon$, satisfy
\begin{align}
  (u_0,u_1),(v_0,v_1)\in H^1\times L^2 \ \ \text{and}\ \ \supp(u_0,u_1,v_0,v_1)\subset[-R,R].
  \label{Eq-Initial-Data}
\end{align}
Let $D(t)$ and $S(t)$ denote the solution operators corresponding to the initial speed in the linear damped wave equation and the linear wave equation, respectively.
For $t\geqslant0$, they are given by
\begin{align*}
  D(t)f(x)&=\frac{1}{2}\mathrm{e}^{-\frac{t}{2}} \int_{x-t}^{x+t}I_0\left(\frac{\sqrt{t^2-(x-y)^2}}{2}\right)f(y)\dd y,\\
  S(t)f(x)&=\frac{1}{2}\int_{x-t}^{x+t}f(y)\dd y,
\end{align*}
where $I_\nu$ denotes the modified Bessel function of the first kind of order $\nu\in\{0,1\}$.
Taking initial data $(u_0,u_1)$ and $(v_0,v_1)$, the homogeneous solutions are
\begin{align*}
  u^{\lin}(t,x)&=\bigl(D(t)+\partial_tD(t)\bigr)u_0(x)+D(t)u_1(x),\\
  v^{\lin}(t,x)&=\partial_tS(t)v_0(x)+S(t)v_1(x).
\end{align*}

\begin{defn}[Mild energy solution]\label{Definition-Solution}
  Let $T>0$ and $\varepsilon>0$.
  A pair $(u,v)$ is called a mild energy solution to the problem \eqref{Eq-Main-System} on $[0,T)$ if
  \begin{align}
    (u,v)\in \bigl(\mathcal C([0,T),H^1)\cap\mathcal C^1([0,T),L^2)\bigr)^2
    \label{Eq-Energy-Class}
  \end{align}
  and it satisfies
  \begin{align}
    u(t,x)&=\varepsilon u^{\lin}(t,x) +\int_0^t D(t-s)|v(s,x)|^p\dd s,\notag\\
    v(t,x)&=\varepsilon v^{\lin}(t,x) +\int_0^t S(t-s)|u(s,x)|^q\dd s.
    \label{Eq-Mild-Solution}
  \end{align}
  The maximal existence time of the solution is denoted by $T_{\varepsilon}$ and is called its lifespan.
\end{defn}

By Proposition~\ref{Proposition-Local}, the initial data in \eqref{Eq-Initial-Data} determine a unique maximal solution in the sense of Definition~\ref{Definition-Solution}.
We next address two main results on the sharp lifespan estimates in this work.

\begin{theorem}[The case $M_v>0$]\label{Theorem-Main}
  Let $p,q>1$ and assume \eqref{Eq-Initial-Data}.
  There exist constants $c>0$ and $\varepsilon_0>0$ such that the lifespan of the unique maximal solution to the problem \eqref{Eq-Main-System} in the sense of Definition~\ref{Definition-Solution} satisfies
  \begin{align}
    T_{\varepsilon}\geqslant c\varepsilon^{-\frac{1}{\gamma_{\neq0}(p,q)}}
    \label{Eq-Lifespan-Lower}
  \end{align}
  for every $\varepsilon\in(0,\varepsilon_0]$.
  If, in addition, the quantities in \eqref{Eq-Initial-Integrals} satisfy
  \begin{align}
    M_u>0\ \ \text{and}\ \ M_v>0,
    \label{Eq-Integral-Signs}
  \end{align}
  then the solution blows up in finite time and
  \begin{align}
    c\varepsilon^{-\frac{1}{\gamma_{\neq0}(p,q)}} \leqslant T_{\varepsilon} \leqslant C\varepsilon^{-\frac{1}{\gamma_{\neq0}(p,q)}}
    \label{Eq-Lifespan-Sharp}
  \end{align}
  for some $C>0$ and all sufficiently small $\varepsilon>0$.
  The constants are independent of $\varepsilon$.
\end{theorem}

\begin{remark}\label{Remark-Lifespan-Comparison}
  For fixed $p,q>1$, let $T_{\varepsilon}^{\mathrm{WW}}$ and $T_{\varepsilon}^{\mathrm{DD}}$ denote the lifespans of the weakly coupled counterparts of the semilinear wave equation \eqref{Eq-Scalar-Wave} and the semilinear damped wave equation \eqref{Eq-Scalar-Damped}, respectively, with nonlinearities $|v|^p$ and $|u|^q$.
  We take the same fixed, nonnegative, smooth, compactly supported initial data in all three systems, with
  \begin{align*}
    \int_{\mb R}u_1(x)\dd x>0 \ \ \text{and}\ \ \int_{\mb R}v_1(x)\dd x>0.
  \end{align*}
  Let us set
  \begin{align*}
    \gamma_{\mathrm{WW}}(p,q):=\frac{2(\max\{p,q\}+1)}{pq-1} \ \ \mbox{and}\ \ \gamma_{\mathrm{DD}}(p,q):=\frac{\max\{p,q\}+1}{pq-1}-\frac{1}{2}.
  \end{align*}
  By applying the test function estimate in \cite[Lemma~3.2]{Ikeda-Sobajima-Wakasa=2019} for the upper bound and estimating the Duhamel terms via the d'Alembert formula for the lower bound, we obtain
  \begin{align*}
    T_{\varepsilon}^{\mathrm{WW}} \approx\varepsilon^{-\frac{1}{\gamma_{\mathrm{WW}}(p,q)}}.
  \end{align*}
  For $p,q\geqslant2$ and $\gamma_{\mathrm{DD}}(p,q)>0$, the estimates in \cite{Nishihara-Wakasugi=2015,Chen-Dao=2023} give
  \begin{align*}
    T_{\varepsilon}^{\mathrm{DD}} \approx\varepsilon^{-\frac{1}{\gamma_{\mathrm{DD}}(p,q)}}.
  \end{align*}
  By considering separately $p\leqslant q$ and $p>q$, we obtain
  \begin{align*}
    \gamma_{\mathrm{DD}}(p,q)<\gamma_{\neq0}(p,q) <\gamma_{\mathrm{WW}}(p,q) \ \ \text{for}\ \ p,q>1.
  \end{align*}
  Combining these inequalities with Theorem~\ref{Theorem-Main} and the lifespan estimates in \cite{Nishihara-Wakasugi=2015,Chen-Dao=2023}, we deduce that, as $\varepsilon\downarrow0$, $\frac{T_{\varepsilon}^{\mathrm{WW}}}{T_{\varepsilon}}\to0$ for $p,q>1$, and $\frac{T_{\varepsilon}}{T_{\varepsilon}^{\mathrm{DD}}}\to0$ if $\gamma_{\mathrm{DD}}(p,q)\geqslant0$.
  Here, $\gamma_{\mathrm{DD}}(p,q)=0$ corresponds to an exponential lifespan for the fully damped system, whereas $\gamma_{\mathrm{DD}}(p,q)<0$ yields small data global in-time existence.
  For $\gamma_{\mathrm{DD}}(p,q)>0$, the inequality
  \begin{align*}
    \frac{1}{\gamma_{\neq0}(p,q)}-\frac{1}{\gamma_{\mathrm{WW}}(p,q)} <\frac{1}{\gamma_{\mathrm{DD}}(p,q)}-\frac{1}{\gamma_{\neq0}(p,q)}
  \end{align*}
  shows that removing $v_t$ from the fully damped system decreases the power of $\varepsilon^{-1}$ in the lifespan estimates more than adding $u_t$ to the undamped system increases it.
\end{remark}

\begin{theorem}[The case $M_v=0$]\label{Theorem-Zero}
  Let $p,q>1$ and assume \eqref{Eq-Initial-Data} as well as
  \begin{align}
    M_v=0.
    \label{Eq-Zero-Mean}
  \end{align}
  There exist constants $c>0$ and $\varepsilon_0>0$ such that the lifespan of the unique maximal solution to the problem \eqref{Eq-Main-System} in the sense of Definition~\ref{Definition-Solution} satisfies
  \begin{align}
    T_{\varepsilon}\geqslant c\varepsilon^{-\frac{1}{\gamma_{=0}(p,q)}}
    \label{Eq-Zero-Lower}
  \end{align}
  for every $\varepsilon\in(0,\varepsilon_0]$.
  If, in addition,
  \begin{align}
    M_u>0,\ \ \int_{\mb R}v_0(x)\dd x\geqslant0 \ \ \text{and}\ \ (v_0,v_1)\not\equiv(0,0),
    \label{Eq-Zero-Signs}
  \end{align}
  then the solution blows up in finite time and
  \begin{align}
    c\varepsilon^{-\frac{1}{\gamma_{=0}(p,q)}} \leqslant T_{\varepsilon} \leqslant C\varepsilon^{-\frac{1}{\gamma_{=0}(p,q)}}
    \label{Eq-Zero-Sharp}
  \end{align}
  for some $C>0$ and all sufficiently small $\varepsilon>0$.
  The constants are independent of $\varepsilon$.
\end{theorem}
\begin{remark}\label{Remark-Initial-Conditions}
  When $M_v>0$, the term $\varepsilon tM_v$ makes the spatial integral of $v$ eventually positive, so no sign condition on $\int_{\mb R}v_0(x)\dd x$ is required.
  However, when $M_v=0$, this term vanishes, and the additional assumptions on $(v_0,v_1)$ in \eqref{Eq-Zero-Signs} provide the positivity and nontriviality needed for the upper lifespan bound.
\end{remark}

To describe the dependence on $M_v$, we write the lifespan powers in \eqref{Eq-Gamma} and \eqref{Eq-Gamma-Zero} as
\begin{align*}
  \alpha_{\neq0}(p,q)&:=\frac{pq-1}{q+2} \ \ \text{and}\ \ \alpha_{=0}(p,q):=\frac{p(pq-1)}{2p+1},
\end{align*}
moreover,
\begin{align}
  \alpha_D(p,q):=\frac{2q(pq-1)}{pq(3-q)+3q+1} \ \ \text{for}\ \ 1<q<3.
  \label{Eq-Lifespan-Powers}
\end{align}
The first two components depend on whether $M_v$ vanishes, whereas $\alpha_D$ arises from the damped component in both cases.
We now rewrite
\begin{align}
  \frac{1}{\gamma_{\neq0}(p,q)}&=
  \begin{cases}
    \min\{\alpha_{\neq0}(p,q),\alpha_D(p,q)\}&\text{if}\ \ 1<q<3,\\
    \alpha_{\neq0}(p,q)&\text{if}\ \ q\geqslant3,
  \end{cases}\notag\\
  \frac{1}{\gamma_{=0}(p,q)}&=
  \begin{cases}
    \min\{\alpha_{=0}(p,q),\alpha_D(p,q)\}&\text{if}\ \ 1<q<3,\\
    \alpha_{=0}(p,q)&\text{if}\ \ q\geqslant3.
  \end{cases}
  \label{Eq-Lifespan-Cases}
\end{align}
By solving $\alpha_{\neq0}(p,q)=\alpha_D(p,q)$ and $\alpha_{=0}(p,q)=\alpha_D(p,q)$ when $1<q<3$, we obtain the transition curves $p=P_{\neq0}(q)$ and $p=P_{=0}(q)$, respectively, in which
\begin{align}
  P_{\neq0}(q):=\frac{2q^2+q-1}{q(3-q)} \ \ \mbox{and}\ \ P_{=0}(q):=\frac{q-1+\sqrt{(q-1)^2+8q^2(3-q)}}{2q(3-q)}.
  \label{Eq-Transition-Curve}
\end{align}
Figures~\ref{Figure-Positive} and \ref{Figure-Zero} show the sharp lifespan estimates in the $(p,q)$-plane for $M_v>0$ and $M_v=0$, respectively.

\begin{figure}[H]
  \centering
  \begin{tikzpicture}
    \begin{axis}[
        width=0.91\textwidth,height=5.9cm,
        xmin=1,xmax=12,ymin=1,ymax=3.65,
        xlabel={$p$},ylabel={$q$},
        ylabel style={rotate=-90},
        xtick={1,3,6,9,12},ytick={1,2,3},
        axis lines=left,tick label style={font=\small},clip=true
      ]
      \addplot[name path=top,draw=none,domain=1:12]{3.65};
      \addplot[name path=bottom,draw=none,domain=1:12]{1};
      \addplot[name path=curve,draw=none,domain=1:12,samples=160] {((3*x-1)+sqrt((3*x-1)^2+4*(x+2)))/(2*(x+2))};
      \addplot[blue!7] fill between[of=top and curve];
      \addplot[orange!12] fill between[of=curve and bottom];
      \addplot[black,thick,domain=1:12,samples=160] {((3*x-1)+sqrt((3*x-1)^2+4*(x+2)))/(2*(x+2))};
      \addplot[densely dashed,gray] coordinates {(1,3) (12,3)};
      \node[font=\small] at (axis cs:6.7,3.30) {$T_{\varepsilon}\approx\varepsilon^{-\frac{pq-1}{q+2}}$};
      \node[font=\small] at (axis cs:7.4,1.45) {$T_{\varepsilon}\approx\varepsilon^{-\frac{2q(pq-1)}{pq(3-q)+3q+1}}$};
      \node[font=\small,anchor=south east] at (axis cs:11.7,2.56) {$p=P_{\neq0}(q)$};
    \end{axis}
  \end{tikzpicture}
  \caption{Sharp lifespan estimates for $M_v>0$.}
  \label{Figure-Positive}
\end{figure}
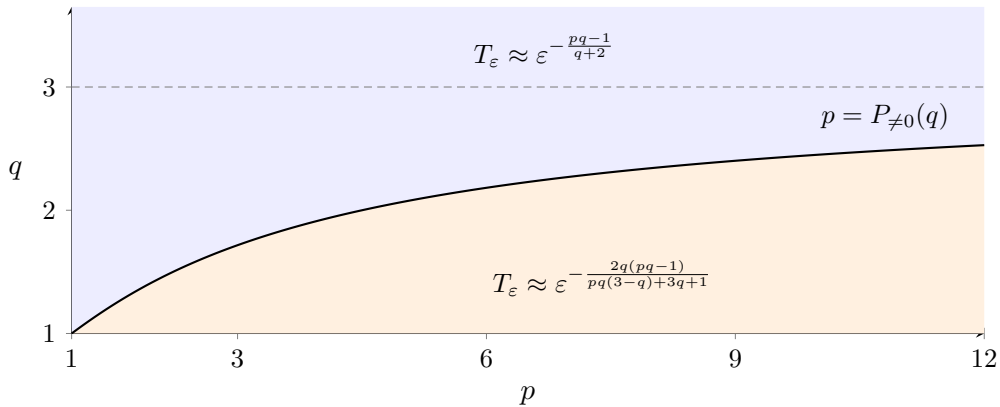

\begin{figure}[H]
  \centering
  \begin{tikzpicture}
    \begin{axis}[
        width=0.91\textwidth,height=5.9cm,
        xmin=1,xmax=12,ymin=1,ymax=3.65,
        xlabel={$p$},ylabel={$q$},
        ylabel style={rotate=-90},
        xtick={1,3,6,9,12},ytick={1,2,3},
        axis lines=left,tick label style={font=\small},clip=true
      ]
      \addplot[name path=top,draw=none,domain=1:12]{3.65};
      \addplot[name path=bottom,draw=none,domain=1:12]{1};
      \addplot[name path=curve,draw=none,domain=1:12,samples=160] {((3*x^2-x-2)+sqrt((3*x^2-x-2)^2+4*x^3))/(2*x^2)};
      \addplot[blue!7] fill between[of=top and curve];
      \addplot[orange!12] fill between[of=curve and bottom];
      \addplot[black,thick,domain=1:12,samples=160] {((3*x^2-x-2)+sqrt((3*x^2-x-2)^2+4*x^3))/(2*x^2)};
      \addplot[densely dashed,gray] coordinates {(1,3) (12,3)};
      \node[font=\small] at (axis cs:6.7,3.30) {$T_{\varepsilon}\approx\varepsilon^{-\frac{p(pq-1)}{2p+1}}$};
      \node[font=\small] at (axis cs:7.4,1.55) {$T_{\varepsilon}\approx\varepsilon^{-\frac{2q(pq-1)}{pq(3-q)+3q+1}}$};
      \node[font=\small,anchor=north east] at (axis cs:11.7,2.73) {$p=P_{=0}(q)$};
    \end{axis}
  \end{tikzpicture}
  \caption{Sharp lifespan estimates for $M_v=0$.}
  \label{Figure-Zero}
\end{figure}
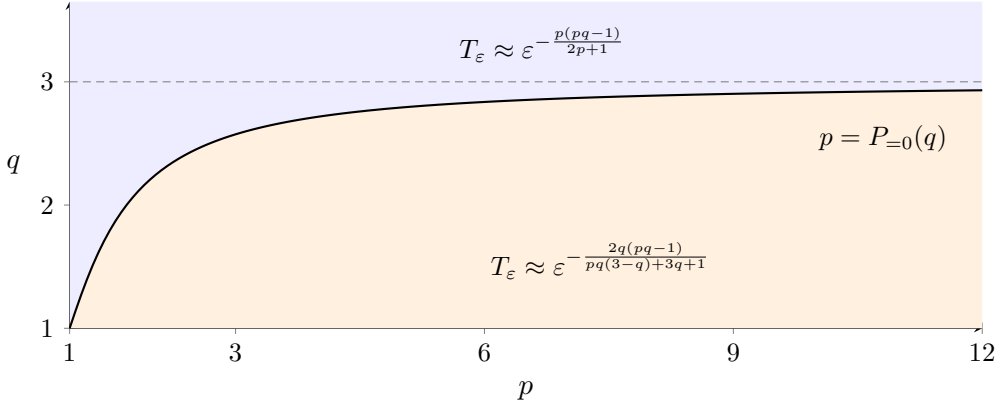

Using \eqref{Eq-Lifespan-Cases} and the ordering $P_{=0}(q)<P_{\neq0}(q)$ of the curves in \eqref{Eq-Transition-Curve} for $1<q<3$, we compare the sharp lifespan powers for $M_v>0$ and $M_v=0$ as follows.
\begin{itemize}
  \item If $q\geqslant3$, or if $1<q<3$ and $1<p\leqslant P_{=0}(q)$, the condition $M_v=0$ increases the sharp lifespan power from $\alpha_{\neq0}$ to $\alpha_{=0}$.

  \item If $1<q<3$ and $P_{=0}(q)<p<P_{\neq0}(q)$, the condition $M_v=0$ increases the sharp lifespan power from $\alpha_{\neq0}$ to $\alpha_D$.
    Here, $\alpha_{\neq0}<\alpha_D<\alpha_{=0}$, and the sharp lifespan power for $M_v=0$ is determined by the damped component.

  \item If $1<q<3$ and $p\geqslant P_{\neq0}(q)$, both classes of initial data have the same sharp lifespan power $\alpha_D$.
    For the single semilinear wave equation in one space dimension, \eqref{Eq-Scalar-Wave-Lifespan} shows that a vanishing integral of the initial speed yields a larger lifespan power than when this integral is nonzero.
    This improvement is absent in the present range of the coupled problem.
\end{itemize}

\begin{remark}\label{Remark-Known-Upper}
  For $M_u,M_v>0$, Theorem~\ref{Theorem-Main} gives a strictly smaller upper lifespan bound than those in \cite{Chen-Reissig=2021,Kita-Kusaba=2022} when $1<q<3$ and $p>P_{\neq0}(q)$.
  In the complementary range, its lower bound matches the known upper bound $T_{\varepsilon}\lesssim \varepsilon^{-\frac{pq-1}{q+2}}$.
\end{remark}

\section{Estimates for the linear problems}\label{Section-Linear}

\subsection{The linear damped wave equation}

We consider the Cauchy problem
\begin{align}
  \begin{cases}
    \phi_{tt}-\phi_{xx}+\phi_t=F(t,x),&(t,x)\in(0,\infty)\times\mb R,\\
    (\phi,\phi_t)(0,x)=(\phi_0,\phi_1)(x),&x\in\mb R.
  \end{cases}
  \label{Eq-Linear-Damped}
\end{align}
By the representation formula in \cite[Lemma~2]{Palmieri-Takamura=2023} with $b=1$ and $m=0$, its solution is given by
\begin{align}
  \phi(t,x)={}&D(t)(\phi_0+\phi_1)(x)+\partial_tD(t)\phi_0(x) +\int_0^tD(t-s)F(s,x)\dd s.
  \label{Eq-Damped-Representation}
\end{align}
Here, we express $D(t)f=K(t,\cdot)*f$, where
\begin{align}
  K(t,z)=\frac{1}{2}\mathrm{e}^{-\frac{t}{2}} I_0\left(\frac{\sqrt{t^2-z^2}}{2}\right) \mathbf1_{\{|z|<t\}}.
  \label{Eq-Damped-Kernel}
\end{align}

\begin{lemma}\label{Lemma-Damped-Estimates}
  For every $r\in[1,\infty]$ and $t\geqslant0$, the operator $D(t)$ preserves nonnegativity and satisfies
  \begin{align}
    \|D(t)f(\cdot)\|_{L^r}\leqslant(1-\mathrm{e}^{-t})\|f\|_{L^r}.
    \label{Damped-Lr-Bounded}
  \end{align}
  Let $(\phi_0,\phi_1)\in H^1\times L^2$ have support in $[-R,R]$.
  For every $1<r<\infty$, the homogeneous solution satisfies
  \begin{align}
    \|\phi^{\lin}(t,\cdot)\|_{L^r} \leqslant C_{R,r}(1+t)^{-\frac{1}{2}(1-\frac{1}{r})} \|(\phi_0,\phi_1)\|_{H^1\times L^2}.
    \label{Damped-Homogeneous-Decay}
  \end{align}
\end{lemma}

\begin{proof}
  Thanks to the fact that $K(t,\cdot)$ is compactly supported, its second spatial derivative has zero integral in the distributional sense.
  Integrating the equation satisfied by $K$ in space, we obtain $m''+m'=0$ for $m(t):=\int_{\mb R}K(t,z)\dd z$, with $m(0)=0$ and $m'(0)=1$, and hence
  \begin{align*}
    \int_{\mb R}K(t,z)\dd z=1-\mathrm{e}^{-t}.
  \end{align*}
  Since $I_0$ is nonnegative on $[0,\infty)$, the kernel $K$ in \eqref{Eq-Damped-Kernel} is nonnegative, and Young's inequality gives \eqref{Damped-Lr-Bounded}.
  Differentiating $D(t)\phi_0$ with respect to $t$, including the endpoints of the integration interval, and substituting into \eqref{Eq-Damped-Representation} with $F=0$, we obtain
  \begin{align}
    \phi^{\lin}(t,x) =\frac{\mathrm{e}^{-\frac t2}}{2} \bigl(\phi_0(x+t)+\phi_0(x-t)\bigr) +\bigl(J(t,\cdot)*\phi_0\bigr)(x)+\bigl(K(t,\cdot)*\phi_1\bigr)(x),
    \label{Damped-Homogeneous-Kernels}
  \end{align}
  where, for $|z|<t$, we set
  \begin{align*}
    J(t,z):=\mathrm{e}^{-\frac{t}{2}} \left(\frac{1}{4}I_0(\rho)+\frac{t}{8\rho}I_1(\rho)\right) \ \ \text{with}\ \ \rho:=\frac{1}{2}\sqrt{t^2-z^2},
  \end{align*}
  and $J(t,z):=0$ for $|z|\geqslant t$.
  Because of $J,K\geqslant0$, the representation \eqref{Damped-Homogeneous-Kernels} gives $\phi^{\lin}(t,x)\geqslant0$ for all $t\geqslant0$ and $x\in\mb R$ whenever $\phi_0,\phi_1\geqslant0$.

  By using the integral representation in \cite[Equation~10.32.1]{Olver-Lozier-Boisvert-Clark=2010}, namely,
  \begin{align*}
    I_0(\rho)=\frac{1}{\pi}\int_0^\pi\mathrm{e}^{\rho\cos\theta}\dd\theta\ \ \mbox{and}\ \ 1-\cos\theta\geqslant\frac{2\theta^2}{\pi^2},
  \end{align*}
  we obtain
  \begin{align*}
    I_0(\rho)\leqslant\frac{\mathrm{e}^{\rho}}{\pi} \int_0^\pi\mathrm{e}^{-\frac{2\rho\theta^2}{\pi^2}}\dd\theta \leqslant C\frac{\mathrm{e}^{\rho}}{\sqrt{1+\rho}}.
  \end{align*}
  The corresponding integral representation of $I_1$ and the power series in \cite[Equations~10.25.2 and~10.32.3]{Olver-Lozier-Boisvert-Clark=2010} give $0\leqslant I_1(\rho)\leqslant I_0(\rho)$ and
  \begin{align*}
    \frac{I_1(\rho)}{\rho} =\frac{1}{2}\sum_{j=0}^{\infty} \frac{(\frac{\rho^2}{4})^j}{j!(j+1)!} \leqslant\frac{1}{2}\sum_{j=0}^{\infty} \frac{(\frac{\rho^2}{4})^j}{(j!)^2} =\frac{1}{2}I_0(\rho),
  \end{align*}
  so for $\rho\geqslant0$ we have
  \begin{align}
    I_0(\rho)\leqslant C\frac{\mathrm{e}^{\rho}}{\sqrt{1+\rho}},\ \ 0\leqslant I_1(\rho)\leqslant I_0(\rho) \ \ \text{and}\ \ \frac{I_1(\rho)}{\rho}\leqslant\frac{1}{2}I_0(\rho).
    \label{Bessel-Elementary-Bounds}
  \end{align}

  For $t\geqslant1$ and $|z|\leqslant\frac{t}{2}$, we have $\rho\approx t$ and $\frac{t}{\rho}\leqslant\frac{4}{\sqrt3}$.
  In view of
  \begin{align*}
    \frac{t}{2}-\rho =\frac{z^2}{2(t+\sqrt{t^2-z^2})}\geqslant\frac{z^2}{4t},
  \end{align*}
  the first two bounds in \eqref{Bessel-Elementary-Bounds} imply
  \begin{align*}
    0\leqslant K(t,z)+J(t,z) \leqslant Ct^{-\frac{1}{2}}\mathrm{e}^{-\frac{z^2}{4t}}.
  \end{align*}
  For $\frac{t}{2}<|z|<t$, the third bound in \eqref{Bessel-Elementary-Bounds} and $\rho\leqslant\frac{\sqrt3t}{4}$ yield
  \begin{align*}
    0\leqslant K(t,z)+J(t,z) \leqslant C(1+t)\mathrm{e}^{-\kappa t} \ \ \text{with}\ \ \kappa=\frac{2-\sqrt3}{4}>0.
  \end{align*}
  Combining these inequalities with $\frac{z^2}{t}\leqslant t$ and $(1+t)\mathrm{e}^{-\frac{\kappa t}{2}}\leqslant Ct^{-\frac{1}{2}}$, we arrive at
  \begin{align}
    0\leqslant K(t,z)+J(t,z) \leqslant Ct^{-\frac{1}{2}}\mathrm{e}^{-c\frac{z^2}{t}}
    \label{Damped-Gaussian-Majorant}
  \end{align}
  for $t\geqslant1$ and $z\in\mb R$.
  Taking $L^r$ norms in \eqref{Damped-Gaussian-Majorant}, we obtain
  \begin{align*}
    \|K(t,\cdot)\|_{L^r}+\|J(t,\cdot)\|_{L^r} \leqslant Ct^{-\frac{1}{2}(1-\frac{1}{r})} \ \ \text{for}\ \ t\geqslant1.
  \end{align*}
  Applying Young's inequality to \eqref{Damped-Homogeneous-Kernels}, we get
  \begin{align*}
    \|\phi^{\lin}(t,\cdot)\|_{L^r} \leqslant{}&\mathrm{e}^{-\frac{t}{2}}\|\phi_0\|_{L^r} +Ct^{-\frac{1}{2}(1-\frac{1}{r})} \bigl(\|\phi_0\|_{L^1}+\|\phi_1\|_{L^1}\bigr).
  \end{align*}
  Compact support and the embedding $H^1\hookrightarrow L^\infty$ then give \eqref{Damped-Homogeneous-Decay} for $t\geqslant1$.

  For $0<t\leqslant1$, both kernels are uniformly bounded and supported in $[-t,t]$, so their $L^a$ norms are uniformly bounded for $1\leqslant a\leqslant\infty$.
  \begin{itemize}
    \item If $1<r<2$, Young's inequality and compact support give
      \begin{align*}
        \|K(t,\cdot)*\phi_1\|_{L^r} \leqslant C\|\phi_1\|_{L^r} \leqslant C_R\|\phi_1\|_{L^2}.
      \end{align*}
    \item If $2\leqslant r<\infty$, applying Young's inequality with $a=\frac{2r}{r+2}\in[1,2)$ gives
      \begin{align*}
        \|K(t,\cdot)*\phi_1\|_{L^r} \leqslant\|K(t,\cdot)\|_{L^a}\|\phi_1\|_{L^2} \leqslant C\|\phi_1\|_{L^2}.
      \end{align*}
  \end{itemize}
  The remaining terms in \eqref{Damped-Homogeneous-Kernels} are bounded by $C\|\phi_0\|_{L^r}\leqslant C_{R,r}\|\phi_0\|_{H^1}$, which proves the estimate for $0<t\leqslant1$.
  The proof is complete.
\end{proof}

\begin{lemma}\label{Lemma-Positive-Diffusion}
  Let $(\phi_0,\phi_1)\in H^1\times L^2$ have support in $[-R,R]$, and set
  \begin{align*}
    M_{\phi}:=\int_{\mb R}\bigl(\phi_0(x)+\phi_1(x)\bigr)\dd x \ \ \text{and}\ \ G_t(x):=\frac{1}{\sqrt{4\pi t}}\mathrm{e}^{-\frac{x^2}{4t}}.
  \end{align*}
  For the homogeneous solution to \eqref{Eq-Linear-Damped}, we have
  \begin{align}
    \sup_{|x|\leqslant\sqrt t} \big|\phi^{\lin}(t,x)-M_{\phi}G_t(x)\big|\leqslant Ct^{-1}
    \label{Eq-Diffusion-Approximation}
  \end{align}
  for all sufficiently large $t$.
  If $M_{\phi}>0$, there exist $c>0$ and $t_0\geqslant1$ such that
  \begin{align}
    \phi^{\lin}(t,x)\geqslant ct^{-\frac{1}{2}} \ \ \text{for}\ \ t\geqslant t_0\ \ \text{and}\ \ |x|\leqslant\sqrt t.
    \label{Eq-Positive-Diffusion}
  \end{align}
\end{lemma}

\begin{proof}
  By expanding $\rho=\frac{1}{2}\sqrt{t^2-z^2}$ as $t\to\infty$, we obtain
  \begin{align*}
    \rho=\frac{t}{2}-\frac{z^2}{4t}+O(t^{-1}) \ \ \text{as well as}\ \ \frac{t}{2\rho}=1+O(t^{-1})
  \end{align*}
  uniformly for $|z|\leqslant\sqrt t+R$.
  Applying the asymptotic formula in \cite[Equation~10.40.1]{Olver-Lozier-Boisvert-Clark=2010}, namely,
  \begin{align*}
    I_\nu(\rho)=\frac{\mathrm{e}^{\rho}}{\sqrt{2\pi\rho}} \bigl(1+O(\rho^{-1})\bigr) \ \ \text{as}\ \ \rho\to\infty,
  \end{align*}
  for $\nu\in\{0,1\}$ to the kernels in \eqref{Eq-Damped-Kernel} and \eqref{Damped-Homogeneous-Kernels}, we arrive at
  \begin{align*}
    K(t,z)=G_t(z)\bigl(1+O(t^{-1})\bigr) \ \ \text{and}\ \ J(t,z)=G_t(z)\bigl(1+O(t^{-1})\bigr)
  \end{align*}
  as $t\to\infty$, uniformly in the same region.
  For $t-\sqrt t>R$ and $|x|\leqslant\sqrt t$, the translated initial displacements in \eqref{Damped-Homogeneous-Kernels} vanish.
  Hence
  \begin{align*}
    \phi^{\lin}(t,x) =\bigl(G_t*(\phi_0+\phi_1)\bigr)(x)+O(t^{-\frac{3}{2}}) \ \ \text{as}\ \ t\to\infty.
  \end{align*}
  By using $\|\partial_xG_t\|_{L^\infty}\leqslant Ct^{-1}$ and compact support, we further obtain
  \begin{align*}
    \big|\bigl(G_t*(\phi_0+\phi_1)\bigr)(x)-M_{\phi}G_t(x)\big| &\leqslant Ct^{-1} \int_{\mb R}|y|\,|\phi_0(y)+\phi_1(y)|\dd y\\
    &\leqslant Ct^{-1},
  \end{align*}
  which proves \eqref{Eq-Diffusion-Approximation}.
  Finally, on $|x|\leqslant\sqrt t$, the following facts:
  \begin{align*}
    G_t(x)\geqslant\frac{\mathrm{e}^{-\frac14}}{\sqrt{4\pi}}t^{-\frac{1}{2}}\ \ \text{and}\ \ M_{\phi}>0
  \end{align*}
  allow us to absorb the $O(t^{-1})$ error for sufficiently large $t$, proving \eqref{Eq-Positive-Diffusion}.
\end{proof}

\subsection{The linear wave equation}

By applying the d'Alembert formula to the Cauchy problem
\begin{align*}
  \begin{cases}
    w_{tt}-w_{xx}=F(t,x),&(t,x)\in(0,\infty)\times\mb R,\\
    (w,w_t)(0,x)=(w_0,w_1)(x),&x\in\mb R,
  \end{cases}
\end{align*}
we obtain
\begin{align}
  w(t,x)={}&\frac{1}{2}\bigl(w_0(x+t)+w_0(x-t)\bigr) +\frac{1}{2}\int_{x-t}^{x+t}w_1(y)\dd y\notag\\
  &+\frac{1}{2}\int_0^t\int_{x-(t-s)}^{x+(t-s)}F(s,y)\dd y\dd s.
  \label{Eq-Wave-Representation}
\end{align}
Taking the $L^\infty$ norm yields
\begin{align}
  \|w(t,\cdot)\|_{L^\infty} \leqslant\|w_0\|_{L^\infty}+\frac{1}{2}\|w_1\|_{L^1} +\frac{1}{2}\int_0^t\|F(s,\cdot)\|_{L^1}\dd s.
  \label{Eq-Wave-Linfty}
\end{align}
In particular, thanks to the initial condition \eqref{Eq-Initial-Data}, we have
\begin{align}
  \sup_{t\geqslant0}\|v^{\lin}(t,\cdot)\|_{L^\infty} \leqslant\|v_0\|_{L^\infty}+\frac{1}{2}\|v_1\|_{L^1}<\infty.
  \label{Eq-Free-Wave-Bound}
\end{align}
The kernels of both $D(t)$ and $S(t)$ also show that these operators have propagation speed one.

If $M_v=0$, let us define
\begin{align}
  A(x):=\frac{v_0(x)-J_1(x)}{2} \ \ \text{and}\ \ B(x):=\frac{v_0(x)+J_1(x)}{2},
  \label{Eq-Wave-Profiles}
\end{align}
where $J_1(x):=\int_{-\infty}^xv_1(y)\dd y$.
Since $J_1,A,B\in H^1$ have support in $[-R,R]$, the formula \eqref{Eq-Wave-Representation} gives
\begin{align}
  v^{\lin}(t,x)=A(x-t)+B(x+t).
  \label{Eq-Travelling-Waves}
\end{align}
The supports of the two terms are disjoint for $t>R$.

\subsection{The damped wave equation with a travelling source}

\begin{lemma}\label{Lemma-Travelling-Source}
  Let $f\geqslant0$ belong to $L^1\cap L^\infty$ and have support in $[-R_f,R_f]$.
  Define
  \begin{align*}
    K_f(t,x):=\int_0^t D(t-s)f(x-s)\dd s.
  \end{align*}
  For every $r\in[1,\infty]$, we have
  \begin{align}
    0\leqslant K_f(t,x)\leqslant\|f\|_{L^1} \ \ \text{and}\ \ \|K_f(t,\cdot)\|_{L^r}\leqslant C_{R_f,r}\|f\|_{L^1}\langle t\rangle^{\frac{1}{r}}.
    \label{Eq-Travelling-Estimate}
  \end{align}
  The same estimates hold for the source $f(x+t)$.
\end{lemma}

\begin{proof}
  For a nonnegative smooth function $f$ with compact support, let us set
  \begin{align*}
    \ml F_f(z):=\int_z^\infty f(y)\dd y \ \ \text{and}\ \ W(t,x):=\ml F_f(x-t).
  \end{align*}
  By differentiation, $W$ satisfies
  \begin{align*}
    \begin{cases}
      W_{tt}-W_{xx}+W_t=f(x-t),&(t,x)\in(0,\infty)\times\mb R,\\
      (W,W_t)(0,x)=(\ml F_f,f)(x),&x\in\mb R.
    \end{cases}
  \end{align*}
  Subtracting the zero-data solution $K_f$ and applying \eqref{Damped-Homogeneous-Kernels}, we obtain
  \begin{align*}
    W(t,x)-K_f(t,x)&=\frac{\mathrm{e}^{-\frac t2}}{2} \bigl(\ml F_f(x+t)+\ml F_f(x-t)\bigr)+\bigl(J(t,\cdot)*\ml F_f\bigr)(x) +\bigl(K(t,\cdot)*f\bigr)(x)\\
    &\geqslant0.
  \end{align*}
  The formula is valid for bounded $\ml F_f$ by localization to the interval $[x-t,x+t]$.
  Since $D(t)$ preserves nonnegativity, we conclude that
  \begin{align*}
    0\leqslant K_f(t,x)\leqslant W(t,x)\leqslant\|f\|_{L^1}.
  \end{align*}
  For general $f$ as in the statement, choose nonnegative smooth $f_n\to f$ in $L^1$ with uniformly bounded support.
  By $0\leqslant K(t,z)\leqslant\frac{1}{2}$, we have
  \begin{align*}
    \sup_{0\leqslant t\leqslant T} \|K_{f_n}(t,\cdot)-K_f(t,\cdot)\|_{L^\infty} \leqslant\frac{T}{2}\|f_n-f\|_{L^1}\to0,
  \end{align*}
  so the pointwise bound passes to the limit.
  Reflection in $x$ gives the same bound for the source $f(x+t)$.
  By finite propagation, we have $\supp K_f(t,\cdot)\subset[-R_f-t,R_f+t]$, and hence
  \begin{align*}
    \|K_f(t,\cdot)\|_{L^r} &\leqslant[2(R_f+t)]^{\frac{1}{r}}\|f\|_{L^1}.\qedhere
  \end{align*}
\end{proof}

By Tonelli's theorem and the spatial integral of $K$, one derives
\begin{align*}
  \int_{\mb R}K_f(t,x)\dd x =(t-1+\mathrm{e}^{-t})\|f\|_{L^1}.
\end{align*}
For $f\not\equiv0$, applying H\"older's inequality on $[-R_f-t,R_f+t]$ therefore gives $\|K_f(t,\cdot)\|_{L^r}\gtrsim t^{\frac{1}{r}}$ for large $t$.
Thus the time power in \eqref{Eq-Travelling-Estimate} is sharp.

\begin{remark}
  By applying the $L^1$--$L^r$ estimate for $D(t)$ to the definition of $K_f$, we obtain
  \begin{align*}
    \|K_f(t,\cdot)\|_{L^r} &\lesssim\|f\|_{L^1} \int_0^t\langle t-s\rangle^{-\frac{1}{2}(1-\frac{1}{r})}\dd s \lesssim\|f\|_{L^1}\langle t\rangle^{\frac{1}{2}(1+\frac{1}{r})},
  \end{align*}
  whereas Lemma~\ref{Lemma-Travelling-Source} gives
  \begin{align*}
    \|K_f(t,\cdot)\|_{L^r}\lesssim\|f\|_{L^1}\langle t\rangle^{\frac{1}{r}}.
  \end{align*}
  Since $\frac{1}{r}<\frac{1}{2}(1+\frac{1}{r})$ for $r>1$, the latter estimate strictly improves the time power in this range.
  The travelling particular solution $\ml F_f(x-t)$ used in the proof gives
  \begin{align*}
    0\leqslant K_f(t,x)\leqslant\ml F_f(x-t)\leqslant\|f\|_{L^1}.
  \end{align*}
  This uniform pointwise bound, together with finite propagation on an interval of length $2(R_f+t)$, yields the power $\frac{1}{r}$.
\end{remark}

When $M_v=0$, the inequality
\begin{align*}
  |A(x-s)+B(x+s)|^p\leqslant2^{p-1}\bigl(|A(x-s)|^p+|B(x+s)|^p\bigr),
\end{align*}
the positivity of $D(t)$ and Lemma~\ref{Lemma-Travelling-Source} yield
\begin{align}
  \left\|\int_0^tD(t-s)|v^{\lin}(s,\cdot)|^p\dd s\right\|_{L^q} &\leqslant C\bigl(\|A\|_{L^p}^p+\|B\|_{L^p}^p\bigr) \langle t\rangle^{\frac{1}{q}}\notag\\
  &\lesssim\langle t\rangle^{\frac{1}{q}},
  \label{Eq-Free-Wave-Damped-Response}
\end{align}
which achieves our aim.

\section{Upper bound estimates for the lifespan}\label{Section-Upper}

In this section, we assume \eqref{Eq-Initial-Data} and either \eqref{Eq-Integral-Signs}, or \eqref{Eq-Zero-Mean} together with \eqref{Eq-Zero-Signs}.
Let $(u,v)$ be the maximal solution to the problem \eqref{Eq-Main-System} in the sense of Definition~\ref{Definition-Solution}.
By Proposition~\ref{Proposition-Local}, we have
\begin{align}
  \supp\bigl(u(t,\cdot),v(t,\cdot)\bigr)\subset[-R-t,R+t] \ \ \text{for}\ \ 0\leqslant t<T_{\varepsilon}.
  \label{Eq-Finite-Propagation}
\end{align}

\subsection{Integral identities and first lower bounds}

Let us define the spatial integrals
\begin{align*}
  \ml U(t):=\int_{\mb R}u(t,x)\dd x \ \ \text{and}\ \ \ml V(t):=\int_{\mb R}v(t,x)\dd x.
\end{align*}
On each compact time interval, testing the equations against a smooth spatial function equal to one on the supports in \eqref{Eq-Finite-Propagation} gives
\begin{align}\label{Eq-Mass-Equations}
  \begin{cases}
    \ml U''(t)+\ml U'(t)=\|v(t,\cdot)\|_{L^p}^p,\\
    \ml V''(t)=\|u(t,\cdot)\|_{L^q}^q,
  \end{cases}
\end{align}
in distributions in time.
By recalling \eqref{Eq-Energy-Class} and \eqref{Eq-Finite-Propagation}, the right-hand sides are continuous, so $\ml U,\ml V\in\mathcal C^2\bigl([0,T_{\varepsilon})\bigr)$.
Writing
\begin{align*}
  m_{u,j}:=\int_{\mb R}u_j(x)\dd x\ \ \mbox{and}\ \ m_{v,j}:=\int_{\mb R}v_j(x)\dd x\ \ \text{for}\ \ j\in\{0,1\},
\end{align*}
we solve \eqref{Eq-Mass-Equations} with the original initial values to obtain
\begin{align}
  \begin{cases}
    \displaystyle{\ml U(t)=\varepsilon\bigl(m_{u,0}+(1-\mathrm{e}^{-t})m_{u,1}\bigr) +\int_0^t\bigl(1-\mathrm{e}^{-(t-s)}\bigr)\|v(s,\cdot)\|_{L^p}^p\dd s,}\\[10pt]
    \displaystyle{\ml V(t)=\varepsilon(m_{v,0}+tM_v) +\int_0^t(t-s)\|u(s,\cdot)\|_{L^q}^q\dd s.}
  \end{cases}
  \label{Eq-Mass-Formulas}
\end{align}
By H\"older's inequality and \eqref{Eq-Finite-Propagation}, we also estimate
\begin{align}
  \begin{cases}
    \|v(t,\cdot)\|_{L^p}^p\geqslant[2(R+t)]^{1-p}|\ml V(t)|^p,\\[0.2em]
    \|u(t,\cdot)\|_{L^q}^q\geqslant[2(R+t)]^{1-q}|\ml U(t)|^q.
  \end{cases}
  \label{Eq-Mass-Holder}
\end{align}

Since $M_u>0$ in both theorems, and either $M_v>0$ or $M_v=0$ with $m_{v,0}\geqslant0$, we may choose $t_0\geqslant\max\{R,1\}$ independently of $\varepsilon$ so that, for $t\geqslant t_0$, one obtains
\begin{align}
  m_{u,0}+(1-\mathrm{e}^{-t})m_{u,1}\geqslant0\ \ \text{and}\ \ m_{v,0}+tM_v\geqslant0.
  \label{Eq-Positive-Linear-Integrals}
\end{align}
Together with \eqref{Eq-Mass-Formulas}, this gives $\ml U(t),\ml V(t)\geqslant0$ for $t_0\leqslant t<T_{\varepsilon}$.

If \eqref{Eq-Integral-Signs} holds, choosing $t_0$ also such that $m_{v,0}+tM_v\geqslant\frac{M_vt}{2}$ leads to
\begin{align*}
  \ml V(t)\geqslant\frac{M_v}{2}\varepsilon t \ \ \text{for}\ \ t_0\leqslant t<T_{\varepsilon}.
\end{align*}

If \eqref{Eq-Zero-Mean} and \eqref{Eq-Zero-Signs} hold, the positivity of the wave Duhamel term implies $v(t,x)\geqslant\varepsilon v^{\lin}(t,x)$.
By \eqref{Eq-Travelling-Waves}, the functions $A(x-t)$ and $B(x+t)$, with $A,B$ defined in \eqref{Eq-Wave-Profiles}, have disjoint supports for $t>R$.
On the set where $A(x-t)>0$, we have $v(t,x)\geqslant\varepsilon A(x-t)>0$, and the same argument applies on the set where $B(x+t)>0$.
With $f_+:=\max\{f,0\}$, we obtain
\begin{align}
  \|v(t,\cdot)\|_{L^p}^p \geqslant\varepsilon^p \bigl(\|A_+\|_{L^p}^p+\|B_+\|_{L^p}^p\bigr) =:c_*\varepsilon^p.
  \label{Eq-Zero-Wave-Source}
\end{align}
If $c_*=0$, then $A,B\leqslant0$ almost everywhere and
\begin{align*}
  0\leqslant\int_{\mb R}v_0(x)\dd x =\int_{\mb R}\bigl(A(x)+B(x)\bigr)\dd x\leqslant0.
\end{align*}
This forces $A=B=0$, hence $v_0=J_1=0$ and $v_1=J_1'=0$, contrary to \eqref{Eq-Zero-Signs}.
Therefore $c_*>0$.
Applying \eqref{Eq-Zero-Wave-Source} and \eqref{Eq-Positive-Linear-Integrals} to the first identity in \eqref{Eq-Mass-Formulas}, by shrinking the domain of integral from $[0,t]$ to $[t_0,t-1]$, we deduce the following lower bound:
\begin{align*}
  \ml U(t)&\geqslant c_*\varepsilon^p \int_{t_0}^{t-1}\bigl(1-\mathrm{e}^{-(t-s)}\bigr)\dd s\\
  &\geqslant c_*(1-\mathrm{e}^{-1})\varepsilon^p(t-1-t_0)\\
  &\geqslant c\varepsilon^p t
\end{align*}
for $2(t_0+1)\leqslant t<T_{\varepsilon}$.
Using \eqref{Eq-Mass-Holder}, we conclude that
\begin{align}
  \|u(t,\cdot)\|_{L^q}^q\geqslant c\varepsilon^{pq}t \ \ \text{for}\ \ t_2\leqslant t<T_{\varepsilon},
  \label{Eq-Zero-u-Source}
\end{align}
where $t_2\geqslant\max\{2(t_0+1),R\}$ is fixed independently of $\varepsilon$.

Under either set of assumptions, Lemmas~\ref{Lemma-Damped-Estimates} and \ref{Lemma-Positive-Diffusion}, applied to the first identity in \eqref{Eq-Mild-Solution}, give
\begin{align*}
  u(t,x)\geqslant\varepsilon u^{\lin}(t,x) \geqslant c\varepsilon t^{-\frac{1}{2}}
\end{align*}
for $t_1\leqslant t<T_{\varepsilon}$ and $|x|\leqslant\sqrt t$, with $t_1$ independent of $\varepsilon$.
Integrating over this spatial interval, we obtain
\begin{align}
  \|u(t,\cdot)\|_{L^q}^q \geqslant\int_{-\sqrt t}^{\sqrt t}|u(t,x)|^q\dd x \geqslant c\varepsilon^q t^{-\frac{q-1}{2}}
  \label{Eq-Diffusion-Source-Lower}
\end{align}
for $t_1\leqslant t<T_{\varepsilon}$.

\subsection{A comparison lemma}

In \cite{Chen-Reissig=2021}, the exponential factor arising from the damping term is treated by an iteration argument with a slicing procedure.
For the inequalities below, restricting time to $[\lambda,2\lambda]$ and rescaling to $[0,1]$ allows us to use a comparison with an ordinary differential system.
Monotone iteration justifies this comparison, while the lifespan estimate follows without a slicing procedure.

\begin{lemma}\label{Lemma-Comparison}
  Let $p,q>1$, $\sigma>0$, $r\in\mb R$, and assume that
  \begin{align*}
    \Gamma_r:=r(pq-1)+q+3-pq>0.
  \end{align*}
  Suppose that there exist $h,c_1,c_2>0$ and $\lambda_0\geqslant1$, independent of $\varepsilon$, such that for every $\lambda\geqslant\lambda_0$ with $2\lambda<T_{\varepsilon}$, the continuous functions $\ml U,\ml V$ are nonnegative on $[\lambda,2\lambda]$ and satisfy
  \begin{align}
    \begin{cases}
      \displaystyle{\ml U(t)\geqslant c_1\lambda^{1-p} \int_\lambda^t\bigl(1-\mathrm{e}^{-(t-s)}\bigr)\ml V(s)^p\dd s,}\\[10pt]
      \displaystyle{\ml V(t)\geqslant h\varepsilon^\sigma\lambda^r +c_2\lambda^{1-q}\int_\lambda^t(t-s)\ml U(s)^q\dd s,}
    \end{cases}
    \label{Eq-Comparison-Assumptions}
  \end{align}
  for $\lambda\leqslant t\leqslant2\lambda$.
  Then, for all sufficiently small $\varepsilon>0$, we have
  \begin{align}
    T_{\varepsilon}\leqslant C\varepsilon^{-\frac{\sigma(pq-1)}{\Gamma_r}}.
    \label{Eq-Comparison-Conclusion}
  \end{align}
\end{lemma}

\begin{proof}
  For convenience, we introduce
  \begin{align}
    \theta:=\frac{\sigma(pq-1)}{\Gamma_r},\ \ \lambda:=L\varepsilon^{-\theta} \ \ \text{and}\ \ \delta:=\varepsilon^\sigma\lambda^r,
    \label{Eq-Comparison-Scaling}
  \end{align}
  where $L\geqslant1$ will be chosen independently of $\varepsilon$.
  For sufficiently small $\varepsilon>0$, we have $\lambda\geqslant\lambda_0$.
  Assuming $2\lambda<T_{\varepsilon}$, we transform $[\lambda,2\lambda]$ into $[0,1]$ by setting
  \begin{align*}
    \widehat U(\tau):= \frac{\ml U\bigl(\lambda(1+\tau)\bigr)}{\delta^p\lambda^{2-p}} \ \ \text{and}\ \ \widehat V(\tau):= \frac{\ml V\bigl(\lambda(1+\tau)\bigr)}{\delta}
  \end{align*}
  for $0\leqslant\tau\leqslant1$.
  The choices in \eqref{Eq-Comparison-Scaling} give
  \begin{align*}
    \delta^{pq-1}\lambda^{q+3-pq} =\varepsilon^{\sigma(pq-1)}\lambda^{\Gamma_r}=L^{\Gamma_r}.
  \end{align*}
  By changing variables in \eqref{Eq-Comparison-Assumptions}, we therefore obtain
  \begin{align*}
    \begin{cases}
      \displaystyle{\widehat U(\tau)\geqslant c_1\int_0^\tau \bigl(1-\mathrm{e}^{-\lambda(\tau-s)}\bigr) \widehat V(s)^p\dd s,}\\[10pt]
      \displaystyle{\widehat V(\tau)\geqslant h+c_2L^{\Gamma_r}\int_0^\tau (\tau-s)\widehat U(s)^q\dd s.}
    \end{cases}
  \end{align*}
  Since $\lambda\geqslant1$, the concavity of $1-\mathrm{e}^{-z}$ implies
  \begin{align*}
    1-\mathrm{e}^{-\lambda z}\geqslant1-\mathrm{e}^{-z} \geqslant(1-\mathrm{e}^{-1})z \ \ \text{for}\ \ 0\leqslant z\leqslant1.
  \end{align*}
  Taking $A:=c_1(1-\mathrm{e}^{-1})$ and $B:=c_2$, we arrive at the following nonlinear integral inequalities:
  \begin{align}
    \begin{cases}
      \displaystyle{ \widehat U(\tau)\geqslant A\int_0^\tau (\tau-s)\widehat V(s)^p\dd s,}\\[10pt]
      \displaystyle{\widehat V(\tau)\geqslant h+BL^{\Gamma_r}\int_0^\tau (\tau-s)\widehat U(s)^q\dd s.}
    \end{cases}
    \label{Eq-Normalized-Comparison}
  \end{align}

  Starting with $\bar{x}_0=0$ and $\bar{y}_0=h$, strongly motivated by \eqref{Eq-Normalized-Comparison}, we define the system
  \begin{align*}
    \begin{cases}
      \displaystyle{ \bar{x}_{n+1}(\tau)=A\int_0^\tau(\tau-s)\bar{y}_n(s)^p\dd s,}\\[10pt]
      \displaystyle{\bar{y}_{n+1}(\tau)=h+BL^{\Gamma_r}\int_0^\tau (\tau-s)\bar{x}_n(s)^q\dd s.}
    \end{cases}
  \end{align*}
  Induction and \eqref{Eq-Normalized-Comparison} immediately give
  \begin{align*}
    0\leqslant\bar{x}_n\leqslant\bar{x}_{n+1}&\leqslant\widehat U,\\
    h\leqslant\bar{y}_n\leqslant\bar{y}_{n+1}&\leqslant\widehat V.
  \end{align*}
  By the continuity of $\ml U,\ml V$ on $[\lambda,2\lambda]$, the rescaled functions $\widehat U,\widehat V$ are continuous and hence bounded on $[0,1]$.
  The monotone sequences $\bar{x}_n,\bar{y}_n$ therefore converge to bounded limits $\bar{x},\bar{y}$, which satisfy the same integral equations by dominated convergence.
  Differentiating these equations, we find that $\bar{x},\bar{y}\in\mathcal C^2([0,1])$ solve
  \begin{align}
    \begin{cases}
      \bar{x}''(\tau)=A\bar{y}(\tau)^p,&\tau\in(0,1),\\
      \bar{y}''(\tau)=BL^{\Gamma_r}\bar{x}(\tau)^q,&\tau\in(0,1),\\
      (\bar{x},\bar{x}')(0)=(0,0),\\
      (\bar{y},\bar{y}')(0)=(h,0).
    \end{cases}
    \label{Eq-Comparison-ODE}
  \end{align}

  To remove the coefficients and normalize the initial value $\bar{y}(0)=h$, we may set $X(s)$ and $Y(s)$ via the next identities:
  \begin{align*}
    \bar{x}(\eta s)=Ah^p\eta^2X(s) \ \ \text{and}\ \ \bar{y}(\eta s)=hY(s),
  \end{align*}
  where
  \begin{align*}
    \eta:=\bigl(BA^qh^{pq-1}L^{\Gamma_r}\bigr)^{-\frac{1}{2q+2}}.
  \end{align*}
  Substituting these expressions into \eqref{Eq-Comparison-ODE} and using $BA^qh^{pq-1}L^{\Gamma_r}\eta^{2q+2}=1$, we obtain
  \begin{align}
    \begin{cases}
      X''(s)=Y(s)^p,&s\in(0,\infty),\\
      Y''(s)=X(s)^q,&s\in(0,\infty),\\
      (X,X')(0)=(0,0),\\
      (Y,Y')(0)=(1,0).
    \end{cases}
    \label{Eq-Fixed-ODE}
  \end{align}
  Let $S_{\mathrm{max}}$ denote the maximal existence time of the solution to \eqref{Eq-Fixed-ODE}.
  By its integral formulas,
  \begin{align*}
    X(s)&=\int_0^s(s-\tau)Y(\tau)^p\dd\tau\geqslant0 \ \ \ \ \ \ \ \ \, \text{and}\ \ X'(s)=\int_0^sY(\tau)^p\dd\tau\geqslant0,\\
    Y(s)&=1+\int_0^s(s-\tau)X(\tau)^q\dd\tau\geqslant1 \ \ \text{and}\ \ Y'(s)=\int_0^sX(\tau)^q\dd\tau\geqslant0.
  \end{align*}
  For a new ansatz
  \begin{align*}
    Z:=X+Y
  \end{align*}
  and $\ell:=\min\{p,q\}>1$, we have $Z\geqslant1$ and $\max\{X,Y\}\geqslant\frac{Z}{2}$.
  Multiplication of the resulting inequality for $Z''$ by $Z'\geqslant0$ and integration yield
  \begin{align*}
    Z''=Y^p+X^q\geqslant c_{p,q}Z^\ell \ \ \Rightarrow\ \ (Z')^2\geqslant\frac{2c_{p,q}}{\ell+1}(Z^{\ell+1}-1),
  \end{align*}
  with $c_{p,q}:=2^{-\max\{p,q\}}$.
  Since $Z''\geqslant c_{p,q}$ and $Z(0)=1$, one has
  \begin{align*}
    Z(s)>1\ \ \text{and}\ \ Z'(s)>0\ \ \text{for}\ \ s>0.
  \end{align*}
  From
  \begin{align*}
    Z'(s)\geqslant\sqrt{\frac{2c_{p,q}}{\ell+1}}\, \sqrt{Z(s)^{\ell+1}-1} \ \ \text{for}\ \ s>0,
  \end{align*}
  we obtain, by separating variables and integrating,
  \begin{align*}
    \frac{Z'(s)}{\sqrt{Z(s)^{\ell+1}-1}} \geqslant\sqrt{\frac{2c_{p,q}}{\ell+1}} \ \ \Rightarrow\ \ s&\leqslant\sqrt{\frac{\ell+1}{2c_{p,q}}} \int_1^{Z(s)}\frac{\dd z}{\sqrt{z^{\ell+1}-1}} \\
    &\leqslant\sqrt{\frac{\ell+1}{2c_{p,q}}} \int_1^\infty\frac{\dd z}{\sqrt{z^{\ell+1}-1}}<\infty.
  \end{align*}
  Thus $S_{\mathrm{max}}<\infty$.
  By the integral formulas and the continuation criterion for ordinary differential equations, $X+Y$ is unbounded as $s\uparrow S_{\mathrm{max}}$.

  Since $S_{\mathrm{max}}$ depends only on $p,q$ and $\Gamma_r>0$, we can choose $L$ sufficiently large that $\eta S_{\mathrm{max}}<1$.
  By the rescaling identities, $\bar{x}(\eta s)+\bar{y}(\eta s)$ is unbounded as $s\uparrow S_{\mathrm{max}}$.
  This contradicts the boundedness of $\bar{x},\bar{y}$ on $[0,1]$.
  We conclude that $T_{\varepsilon}\leqslant2\lambda$, which proves \eqref{Eq-Comparison-Conclusion}.
\end{proof}

\subsection{Proof of the upper bounds}

\noindent\emph{The case $M_v>0$.}
Under \eqref{Eq-Integral-Signs}, choose $\lambda\geqslant\max\{t_0,R,1\}$ with $2\lambda<T_{\varepsilon}$.
Since $2(R+s)\leqslant6\lambda$ for $\lambda\leqslant s\leqslant2\lambda$, the equations \eqref{Eq-Mass-Formulas}--\eqref{Eq-Positive-Linear-Integrals} give
\begin{align*}
  \begin{cases}
    \displaystyle{ \ml U(t)\geqslant c_1\lambda^{1-p} \int_\lambda^t\bigl(1-\mathrm{e}^{-(t-s)}\bigr)\ml V(s)^p\dd s,}\\[10pt]
    \displaystyle{\ml V(t)\geqslant\frac{M_v}{2}\varepsilon\lambda +c_2\lambda^{1-q}\int_\lambda^t(t-s)\ml U(s)^q\dd s,}
  \end{cases}
\end{align*}
for $\lambda\leqslant t\leqslant2\lambda$, with $c_1,c_2$ independent of $\lambda$.
Applying Lemma~\ref{Lemma-Comparison} with $\sigma=r=1$ and $\Gamma_1=q+2$, we obtain
\begin{align}
  T_{\varepsilon}\leqslant C\varepsilon^{-\frac{pq-1}{q+2}} \ \ \text{for}\ \ p,q>1.
  \label{Eq-Upper-Wave}
\end{align}

\medskip\noindent\emph{The case $M_v=0$.}
Under \eqref{Eq-Zero-Mean} and \eqref{Eq-Zero-Signs}, choose $\lambda\geqslant4\max\{t_0,t_2,R,1\}$ with $2\lambda<T_{\varepsilon}$.
For $\lambda\leqslant t\leqslant2\lambda$, the inequality for $\ml U$ remains valid.
In the identity for $\ml V$, we retain the integrals over the disjoint intervals $[\frac{\lambda}{4},\frac{\lambda}{2}]$ and $[\lambda,t]$.
The equations \eqref{Eq-Zero-u-Source} and \eqref{Eq-Mass-Holder}, together with $t-s\geqslant\frac{\lambda}{2}$ and $s\geqslant\frac{\lambda}{4}$ on the first interval, yield
\begin{align*}
  \ml V(t)&\geqslant c\varepsilon^{pq} \int_{\frac{\lambda}{4}}^{\frac{\lambda}{2}}(t-s)s\dd s +c_2\lambda^{1-q}\int_\lambda^t(t-s)\ml U(s)^q\dd s\\
  &\geqslant h\varepsilon^{pq}\lambda^3 +c_2\lambda^{1-q}\int_\lambda^t(t-s)\ml U(s)^q\dd s.
\end{align*}
Applying Lemma~\ref{Lemma-Comparison} with $\sigma=pq$, $r=3$ and $\Gamma_3=q(2p+1)>0$, we get
\begin{align}
  T_{\varepsilon}\leqslant C\varepsilon^{-\frac{p(pq-1)}{2p+1}} \ \ \text{for}\ \ p,q>1.
  \label{Eq-Upper-Zero}
\end{align}

\medskip\noindent\emph{The contribution of the damped component.}
Under either set of data assumptions, let $1<q<3$ and choose $\lambda\geqslant4\max\{t_0,t_1,R,1\}$ with $2\lambda<T_{\varepsilon}$.
For $\lambda\leqslant t\leqslant2\lambda$, we again split the integral for $\ml V$ over $[\frac{\lambda}{4},\frac{\lambda}{2}]$ and $[\lambda,t]$.
Since $t-s\geqslant\frac{\lambda}{2}$ on the first interval, whose length is $\frac{\lambda}{4}$, the equations \eqref{Eq-Mass-Holder}, \eqref{Eq-Positive-Linear-Integrals} and \eqref{Eq-Diffusion-Source-Lower} give
\begin{align*}
  \ml V(t)&\geqslant c\varepsilon^q \int_{\frac{\lambda}{4}}^{\frac{\lambda}{2}}(t-s)s^{-\frac{q-1}{2}}\dd s +c_2\lambda^{1-q}\int_\lambda^t(t-s)\ml U(s)^q\dd s\\
  &\geqslant h\varepsilon^q\lambda^{\frac{5-q}{2}} +c_2\lambda^{1-q}\int_\lambda^t(t-s)\ml U(s)^q\dd s.
\end{align*}
The inequality for $\ml U$ follows from \eqref{Eq-Mass-Formulas} as before.
Applying Lemma~\ref{Lemma-Comparison} with $\sigma=q$, $r=\frac{5-q}{2}$ and $\Gamma_r=\frac{pq(3-q)+3q+1}{2}>0$, we conclude that
\begin{align}
  T_{\varepsilon}\leqslant C\varepsilon^{-\frac{2q(pq-1)}{pq(3-q)+3q+1}} \ \ \text{for}\ \ p>1,\ 1<q<3.
  \label{Eq-Upper-Diffusion}
\end{align}
Combining \eqref{Eq-Upper-Diffusion} with \eqref{Eq-Upper-Wave} or \eqref{Eq-Upper-Zero} proves the upper bounds in Theorems~\ref{Theorem-Main} and~\ref{Theorem-Zero}.

\section{Local in-time existence and lower bound estimates}\label{Section-Lower}

\subsection{Local in-time well-posedness and continuation}

To derive the lifespan lower bounds from the $L^\infty$ estimates for $v$ below, we establish a continuation criterion involving only this component.
By finite propagation and the energy estimates, boundedness of $\|v(t,\cdot)\|_{L^\infty}$ up to a finite time bounds the energy norms of both components and hence permits continuation.

\begin{prop}\label{Proposition-Local}
  Let $p,q>1$ and assume \eqref{Eq-Initial-Data}.
  For every $\varepsilon>0$, the problem \eqref{Eq-Main-System} admits a unique maximal solution in the sense of Definition~\ref{Definition-Solution}.
  The solution depends continuously on the initial data and satisfies \eqref{Eq-Finite-Propagation}.
  If $T_{\varepsilon}<\infty$, then
  \begin{align}
    \limsup_{t\uparrow T_{\varepsilon}} \|v(t,\cdot)\|_{L^\infty}=\infty.
    \label{Eq-Continuation-Criterion}
  \end{align}
\end{prop}

\begin{proof}
  For $\mu\in\{0,1\}$, let us consider
  \begin{align*}
    \begin{cases}
      w_{tt}-w_{xx}+\mu w_t=F(t,x),&(t,x)\in(0,T)\times\mb R,\\
      (w,w_t)(0,x)=(w_0,w_1)(x),&x\in\mb R.
    \end{cases}
  \end{align*}
  For smooth data, multiplication by $w_t$ and integration in $x$ give
  \begin{align*}
    E_w'(t)=-\mu\|w_t(t,\cdot)\|_{L^2}^2 +(w(t,\cdot),w_t(t,\cdot))_{L^2} +(F(t,\cdot),w_t(t,\cdot))_{L^2},
  \end{align*}
  where
  \begin{align*}
    E_w(t):=\frac{1}{2}\bigl(\|w_t(t,\cdot)\|_{L^2}^2+ \|w_x(t,\cdot)\|_{L^2}^2+\|w(t,\cdot)\|_{L^2}^2\bigr).
  \end{align*}
  By the Cauchy--Schwarz and Young inequalities, we have
  \begin{align*}
    E_w'(t) &\leqslant\|w(t,\cdot)\|_{L^2}\|w_t(t,\cdot)\|_{L^2} +\|F(t,\cdot)\|_{L^2}\|w_t(t,\cdot)\|_{L^2}\\
    &\leqslant E_w(t)+\sqrt{2E_w(t)}\,\|F(t,\cdot)\|_{L^2}.
  \end{align*}
  For $\delta>0$, this implies
  \begin{align*}
    \frac{\mathrm{d}}{\mathrm{d}t}\sqrt{E_w(t)+\delta} =\frac{E_w'(t)}{2\sqrt{E_w(t)+\delta}} \leqslant\frac{1}{2}\sqrt{E_w(t)+\delta} +\frac{1}{\sqrt2}\|F(t,\cdot)\|_{L^2}.
  \end{align*}
  Using the integrating factor $\mathrm{e}^{-\frac{t}{2}}$, integrating over $[0,t]$, and letting $\delta\downarrow0$, we obtain
  \begin{align*}
    \sqrt{E_w(t)} &\leqslant\mathrm{e}^{\frac{t}{2}}\sqrt{E_w(0)} +\frac{1}{\sqrt2}\int_0^t\mathrm{e}^{\frac{t-s}{2}} \|F(s,\cdot)\|_{L^2}\dd s\\
    &\leqslant\mathrm{e}^{\frac{T}{2}}\left(\sqrt{E_w(0)} +\frac{1}{\sqrt2}\int_0^t\|F(s,\cdot)\|_{L^2}\dd s\right).
  \end{align*}
  By the definition of $E_w$, this yields
  \begin{align}
    \|(w,w_t)(t,\cdot)\|_{H^1\times L^2} \leqslant C_T\left(\|(w_0,w_1)\|_{H^1\times L^2} +\int_0^t\|F(s,\cdot)\|_{L^2}\dd s\right)
    \label{Eq-Linear-Energy}
  \end{align}
  for $0\leqslant t\leqslant T$.
  By approximation, the same estimate holds for energy data and $F\in L^1([0,T],L^2)$, and the solution belongs to $\mathcal C([0,T],H^1)\cap\mathcal C^1([0,T],L^2)$.
  For $T\leqslant1$, the constant can be chosen independently of $T$.

  By the mean value inequality and $H^1\hookrightarrow L^\infty$, we have
  \begin{align}
    \big\||f|^\ell-|g|^\ell\big\|_{L^2} &\leqslant C_\ell \bigl(\|f\|_{L^\infty}+\|g\|_{L^\infty}\bigr)^{\ell-1}\|f-g\|_{L^2} \notag\\
    &\leqslant C_\ell \bigl(\|f\|_{H^1}+\|g\|_{H^1}\bigr)^{\ell-1}\|f-g\|_{H^1}
    \label{Eq-Nonlinear-Lipschitz}
  \end{align}
  for $\ell>1$ and $f,g\in H^1$.
  On
  \begin{align*}
    \mathcal X_T:=\bigl(\mathcal C([0,T],H^1)\cap \mathcal C^1([0,T],L^2)\bigr)^2,
  \end{align*}
  we use the norm
  \begin{align*}
    \|(u,v)\|_{\mathcal X_T}:= \sup_{0\leqslant t\leqslant T} \bigl(\|(u,u_t)(t,\cdot)\|_{H^1\times L^2} +\|(v,v_t)(t,\cdot)\|_{H^1\times L^2}\bigr).
  \end{align*}
  Let $M\geqslant1$ satisfy
  \begin{align*}
    \varepsilon\left(\|(u_0,u_1)\|_{H^1\times L^2} +\|(v_0,v_1)\|_{H^1\times L^2}\right)\leqslant M.
  \end{align*}
  By \eqref{Eq-Linear-Energy} and \eqref{Eq-Nonlinear-Lipschitz}, the right-hand sides of \eqref{Eq-Mild-Solution} define a map $\mathcal N$ satisfying, on a ball of radius $B$ in $\mathcal X_T$,
  \begin{align*}
    \|\mathcal N(u,v)\|_{\mathcal X_T} &\leqslant CM+CT(B^p+B^q),\\
    \|\mathcal N(u,v)-\mathcal N(\widetilde u,\widetilde v)\|_{\mathcal X_T} &\leqslant CT(B^{p-1}+B^{q-1}) \|(u-\widetilde u,v-\widetilde v)\|_{\mathcal X_T}.
  \end{align*}
  With $B=2CM$, choosing $c>0$ sufficiently small makes this map a contraction from the ball into itself whenever
  \begin{align}
    0<T\leqslant T_M:=\frac{c}{1+M^{p-1}+M^{q-1}}\leqslant1.
    \label{Eq-Uniform-Local-Time}
  \end{align}
  The contraction principle proves local in-time existence, uniqueness and continuous dependence, as also obtained in \cite[Proposition~2.2]{Kita-Kusaba=2022}.

  Since $T_M$ in \eqref{Eq-Uniform-Local-Time} is uniform over initial states with energy norm at most $M$ and the equations are autonomous, restarting the solution at a time less than $T_M$ before a finite maximal time gives a continuation whenever its energy stays bounded.
  For compactly supported initial data, the propagation speed of $D(t)$ and $S(t)$ shows inductively that the Picard iterates are supported in $[-R-t,R+t]$.
  Passing to the energy limit gives \eqref{Eq-Finite-Propagation}.

  Suppose now that $T_{\varepsilon}<\infty$ and that $\|v(t,\cdot)\|_{L^\infty}\leqslant M_\infty$ for $t<T_{\varepsilon}$.
  Finite propagation yields
  \begin{align*}
    \big\||v(t,\cdot)|^p\big\|_{L^2} \leqslant[2(R+T_{\varepsilon})]^{\frac{1}{2}}M_\infty^p.
  \end{align*}
  Applying \eqref{Eq-Linear-Energy} to the damped wave equation with the fixed time $T=T_{\varepsilon}$ therefore bounds $\|(u,u_t)(t,\cdot)\|_{H^1\times L^2}$ uniformly.
  By the embedding $H^1\hookrightarrow L^{2q}$, we also have
  \begin{align*}
    \big\||u(t,\cdot)|^q\big\|_{L^2} =\|u(t,\cdot)\|_{L^{2q}}^q\leqslant C.
  \end{align*}
  A second application of \eqref{Eq-Linear-Energy} bounds the energy of $(v,v_t)$, contradicting maximality by the continuation argument above.
  This proves \eqref{Eq-Continuation-Criterion}.
\end{proof}

\subsection{Estimates for the nonlinear terms}\label{Subsection-Lower-Proof}

Under \eqref{Eq-Initial-Data}, we write
\begin{align}
  u(t,x)=\varepsilon u^{\lin}(t,x)+u^{\nlin}(t,x) \ \ \text{and}\ \ v(t,x)=\varepsilon v^{\lin}(t,x)+v^{\nlin}(t,x),
  \label{Eq-Lower-Decomposition}
\end{align}
where $u^{\nlin}$ and $v^{\nlin}$ are the Duhamel terms in \eqref{Eq-Mild-Solution}.
Fix $T>0$ and put $S:=\min\{T,T_{\varepsilon}\}$.
All estimates below hold for $0\leqslant t<S$.
With the notation
\begin{align}
  H_q(t):=\int_0^t\langle s\rangle^{-\frac{q-1}{2}}\dd s,
  \label{Eq-Hq}
\end{align}
the estimates \eqref{Eq-Wave-Linfty} and \eqref{Damped-Homogeneous-Decay} imply
\begin{align}
  \|v^{\nlin}(t,\cdot)\|_{L^\infty} &\lesssim\int_0^t\|u(s,\cdot)\|_{L^q}^q\dd s\notag\\
  &\lesssim\varepsilon^qH_q(t) +\int_0^t\|u^{\nlin}(s,\cdot)\|_{L^q}^q\dd s.
  \label{Eq-Lower-v}
\end{align}
By finite propagation and \eqref{Damped-Lr-Bounded}, we also have
\begin{align}
  \|u^{\nlin}(t,\cdot)\|_{L^q} &\leqslant\int_0^t\big\||v(s,\cdot)|^p\big\|_{L^q}\dd s\notag\\
  &\lesssim\int_0^t\langle s\rangle^{\frac{1}{q}} \|v(s,\cdot)\|_{L^\infty}^p\dd s.
  \label{Eq-Lower-u}
\end{align}
Let us define
\begin{align}
  W_{\neq0}(t):=\|v(t,\cdot)\|_{L^\infty} \ \ \text{and}\ \ \Phi_{\neq0}(t):=\varepsilon+\varepsilon^qH_q(t).
  \label{Eq-Lower-Weight-General}
\end{align}
Combining \eqref{Eq-Free-Wave-Bound}, \eqref{Eq-Lower-v} and \eqref{Eq-Lower-u}, we obtain
\begin{align}
  W_{\neq0}(t)\leqslant C_0\Phi_{\neq0}(t) +C_1\int_0^t\left(\int_0^s \langle\tau\rangle^{\frac{1}{q}}W_{\neq0}(\tau)^p\dd\tau\right)^q\dd s.
  \label{Eq-Lower-Scalar-General}
\end{align}

If $M_v=0$, the decomposition \eqref{Eq-Lower-Decomposition}, the positivity of $D(t)$ and the inequality
\begin{align*}
  |v|^p\leqslant2^{p-1}(\varepsilon^p|v^{\lin}|^p+|v^{\nlin}|^p)
\end{align*}
allow us to apply \eqref{Eq-Free-Wave-Damped-Response} to the first term.
Estimating the second term by finite propagation, we get
\begin{align}
  \|u^{\nlin}(t,\cdot)\|_{L^q} \lesssim\varepsilon^p\langle t\rangle^{\frac{1}{q}} +\int_0^t\langle s\rangle^{\frac{1}{q}} \|v^{\nlin}(s,\cdot)\|_{L^\infty}^p\dd s.
  \label{Eq-Lower-u-Zero}
\end{align}
Using \eqref{Eq-Lower-v}, \eqref{Eq-Lower-u-Zero} and $\int_0^t\langle s\rangle\dd s\leqslant\frac{\langle t\rangle^2}{2}$, and setting
\begin{align}
  W_{=0}(t):=\|v^{\nlin}(t,\cdot)\|_{L^\infty} \ \ \text{and}\ \ \Phi_{=0}(t):=\varepsilon^qH_q(t) +\varepsilon^{pq}\langle t\rangle^2,
  \label{Eq-Lower-Weight-Zero}
\end{align}
we obtain
\begin{align}
  W_{=0}(t)\leqslant C_0\Phi_{=0}(t) +C_1\int_0^t\left(\int_0^s \langle\tau\rangle^{\frac{1}{q}}W_{=0}(\tau)^p\dd\tau\right)^q\dd s.
  \label{Eq-Lower-Scalar-Zero}
\end{align}
The constants $C_0\geqslant1$ and $C_1>0$ can be chosen the same in \eqref{Eq-Lower-Scalar-General} and \eqref{Eq-Lower-Scalar-Zero}, independently of $T$ and $\varepsilon$.

\makeatletter
\@secpenalty=0
\makeatother
\subsection{A common continuity argument}
\makeatletter
\@secpenalty=-300
\makeatother

Let $(W,\Phi)$ be either $(W_{\neq0},\Phi_{\neq0})$ or $(W_{=0},\Phi_{=0})$, using the second pair when $M_v=0$.
Define
\begin{align}
  \Psi(t):=\int_0^t\langle s\rangle^{\frac{1}{q}}\Phi(s)^p\dd s \ \ \text{and}\ \ Q(T):=\sup_{0<t\leqslant T} \frac{1}{\Phi(t)}\int_0^t\Psi(s)^q\dd s.
  \label{Eq-Lower-Q}
\end{align}
As both weights are positive and nondecreasing, we have
\begin{align*}
  \Psi(s)\leqslant C\langle s\rangle^{1+\frac{1}{q}}\Phi(s)^p.
\end{align*}
For $0<t\leqslant T$, this gives
\begin{align*}
  \frac{1}{\Phi(t)}\int_0^t\Psi(s)^q\dd s &\leqslant C\Phi(t)^{pq-1}\int_0^t\langle s\rangle^{q+1}\dd s\\
  &\leqslant C\langle t\rangle^{q+2}\Phi(t)^{pq-1},
\end{align*}
and consequently
\begin{align}
  Q(T)\leqslant C\langle T\rangle^{q+2}\Phi(T)^{pq-1}.
  \label{Eq-Monotone-Q}
\end{align}
We write $Q_{\neq0}$ and $Q_{=0}$ for the quantities corresponding to $\Phi_{\neq0}$ and $\Phi_{=0}$, respectively.

Choose $M=4C_0$ and $\delta_0>0$ such that $C_1M^{pq}\delta_0\leqslant C_0$.
If $Q(T)\leqslant\delta_0$, then on any interval where $W\leqslant M\Phi$, the corresponding inequality \eqref{Eq-Lower-Scalar-General} or \eqref{Eq-Lower-Scalar-Zero} yields
\begin{align*}
  W(t)\leqslant \bigl(C_0+C_1M^{pq}Q(T)\bigr)\Phi(t) \leqslant2C_0\Phi(t).
\end{align*}
At $t=0$, we have
\begin{align*}
  W_{\neq0}(0)\leqslant C_0\Phi_{\neq0}(0)
\end{align*}
after enlarging $C_0$, while
\begin{align*}
  W_{=0}(0)=0<\Phi_{=0}(0)=\varepsilon^{pq}.
\end{align*}
Thus $W<M\Phi$ initially, and continuity excludes a first time when equality holds.
It follows that $W\leqslant2C_0\Phi$ on $[0,S)$.
\begin{itemize}
  \item For $W=W_{\neq0}$, the definition in \eqref{Eq-Lower-Weight-General} and the monotonicity of $H_q$ give
    \begin{align*}
      \|v(t,\cdot)\|_{L^\infty} &=W_{\neq0}(t)\leqslant2C_0\Phi_{\neq0}(t)\\
      &\leqslant2C_0\bigl(\varepsilon+\varepsilon^qH_q(T)\bigr)<\infty \ \ \text{for}\ \ 0\leqslant t<S.
    \end{align*}
  \item For $W=W_{=0}$, the estimate \eqref{Eq-Free-Wave-Bound} and the monotonicity of $\Phi_{=0}$ give
    \begin{align*}
      \|v(t,\cdot)\|_{L^\infty}&\leqslant \varepsilon\|v^{\lin}(t,\cdot)\|_{L^\infty}+W_{=0}(t)\\
      &\leqslant C\varepsilon+2C_0\Phi_{=0}(T)<\infty \ \ \text{for}\ \ 0\leqslant t<S.
    \end{align*}
\end{itemize}
If $T_{\varepsilon}\leqslant T$, either bound contradicts \eqref{Eq-Continuation-Criterion}.
As a consequence,
\begin{align}
  Q(T)\leqslant\delta_0\ \ \Rightarrow\ \ T_{\varepsilon}>T.
  \label{Eq-Lower-Continuation}
\end{align}

\subsection{Proof of the lower lifespan bounds}\label{Subsection-Lower-Powers}

\noindent\emph{The range $1<q<3$.}
Let $d:=\frac{3-q}{2}>0$ and $\beta:=q+2+d(pq-1)$.
By using
\begin{align*}
  H_q(t)=\frac{\langle t\rangle^d-1}{d}\lesssim\langle t\rangle^d
\end{align*}
and \eqref{Eq-Monotone-Q}, we obtain
\begin{align}
  Q_{\neq0}(T)&\lesssim \varepsilon^{pq-1}\langle T\rangle^{q+2} +\varepsilon^{q(pq-1)}\langle T\rangle^\beta,\notag\\
  Q_{=0}(T)&\lesssim \varepsilon^{pq(pq-1)}\langle T\rangle^{q(2p+1)} +\varepsilon^{q(pq-1)}\langle T\rangle^\beta,
  \label{Eq-Lower-Q-Estimates}
\end{align}
where we used $(a+b)^\ell\leqslant C_\ell(a^\ell+b^\ell)$ for $a,b\geqslant0$ and $\ell=pq-1>0$.
Here,
\begin{align*}
  q+2+2(pq-1)=q(2p+1)\ \ \mbox{and} \ \ \beta=\frac{pq(3-q)+3q+1}{2}.
\end{align*}

Take $(\alpha_*,k)=(\alpha_{\neq0},q+2)$ for the first inequality and $(\alpha_*,k)=(\alpha_{=0},q(2p+1))$ for the second one, with $\alpha_{\neq0}$ and $\alpha_{=0}$ defined in Section~\ref{Section-Main-Result}, and $\alpha_D$ defined in \eqref{Eq-Lifespan-Powers}.
Setting $\alpha:=\min\{\alpha_*,\alpha_D\}$ and $T=c\varepsilon^{-\alpha}$, choose $c\in(0,1]$ and then $\varepsilon$ sufficiently small that $T\geqslant1$ and $\langle T\rangle\leqslant2T$.
The respective estimate in \eqref{Eq-Lower-Q-Estimates} becomes
\begin{align*}
  Q(T)\leqslant C\left(c^k\varepsilon^{k(\alpha_*-\alpha)} +c^\beta\varepsilon^{\beta(\alpha_D-\alpha)}\right) \leqslant C(c^k+c^\beta).
\end{align*}
Choosing $c$ sufficiently small gives $Q(T)\leqslant\delta_0$, including the case $\alpha_* =\alpha_D$ on either transition curve.

\medskip\noindent\emph{The range $q\geqslant3$.}
By \eqref{Eq-Hq}, $H_q$ is bounded for $q>3$, while $H_3(t)=\ln\langle t\rangle$.
For the weight $\Phi_{\neq0}$, take $T=c\varepsilon^{-\alpha_{\neq0}}$.
For sufficiently small $\varepsilon>0$, we have
\begin{align*}
  \varepsilon^{q-1}H_q(T)\leqslant1
\end{align*}
and hence $\Phi_{\neq0}(t)\lesssim\varepsilon$ for $0\leqslant t\leqslant T$.
Applying \eqref{Eq-Monotone-Q}, we get
\begin{align*}
  Q_{\neq0}(T)\lesssim\varepsilon^{pq-1}\langle T\rangle^{q+2} \leqslant Cc^{q+2}.
\end{align*}

For $M_v=0$, take $T=c\varepsilon^{-\alpha_{=0}}$.
The equations \eqref{Eq-Lower-Weight-Zero} and \eqref{Eq-Monotone-Q} give
\begin{align*}
  Q_{=0}(T)\lesssim \varepsilon^{pq(pq-1)}\langle T\rangle^{q(2p+1)} +\varepsilon^{q(pq-1)}\langle T\rangle^{q+2} \bigl(1+\ln\langle T\rangle\bigr)^{pq-1}.
\end{align*}
The first term is at most $Cc^{q(2p+1)}$.
By the identity
\begin{align*}
  q(pq-1)-(q+2)\alpha_{=0} =\frac{(pq-1)\bigl(p(q-2)+q\bigr)}{2p+1}>0,
\end{align*}
the second term is arbitrarily small for sufficiently small $\varepsilon>0$ once $c$ is fixed.
Choosing $c$ sufficiently small and then $\varepsilon$ sufficiently small proves $Q_{=0}(T)\leqslant\delta_0$, also at $q=3$.

By \eqref{Eq-Lifespan-Cases} and \eqref{Eq-Lower-Continuation}, these choices of $T$ prove \eqref{Eq-Lifespan-Lower} and \eqref{Eq-Zero-Lower}.
Together with the upper bounds from Section~\ref{Section-Upper}, these yield \eqref{Eq-Lifespan-Sharp} and \eqref{Eq-Zero-Sharp}, completing the proofs of Theorems~\ref{Theorem-Main} and \ref{Theorem-Zero}.

\section*{Acknowledgments}
Wenhui Chen is supported in part by the National Natural Science Foundation of China (grant No. 12301270) and the Guangdong Basic and Applied Basic Research Foundation (grant No. 2025A1515010240).
The author would like to thank Alessandro Palmieri for suggesting the topic investigated in this paper.

\end{document}